\documentclass[12pt]{amsart}
\usepackage{amsthm,amssymb,amsmath, amsfonts,enumerate}
\usepackage[numbers,sort&compress]{natbib}
\usepackage{hyperref}
\hypersetup{
pdftitle=   {Defective colouring of graphs excluding a subgraph or minor},
pdfauthor=  {Patrice Ossona de Mendez, Sang-il Oum, David R. Wood}}
\usepackage[a4paper,margin=30mm]{geometry}

\usepackage{tikz}

\usepackage[noabbrev,capitalise]{cleveref}
\crefname{LEM}{Lemma}{Lemmas}
\crefname{THM}{Theorem}{Theorems}
\crefname{PROP}{Proposition}{Propositions}
\crefname{CON}{Conjecture}{Conjectures}
\crefname{CLAIM}{Claim}{Claims}

\newtheorem{THM}{Theorem}[section]
\newtheorem{LEM}[THM]{Lemma}
\newtheorem{COR}[THM]{Corollary}
\newtheorem{PROP}[THM]{Proposition}
\newtheorem{CON}[THM]{Conjecture}
\newtheorem{CLAIM}{Claim}
\theoremstyle{remark}
\newtheorem*{REM}{Remark}
\theoremstyle{definition}

\newcommand\abs[1]{\lvert #1\rvert}

\newcommand{\floor}[1]{\lfloor{#1}\rfloor}

\newcommand\tnb{\,\widetilde\triangledown\,}

\newcommand\nb[1][1/2]{\widetilde{\nabla}_{#1}}
\DeclareMathOperator{\td}{td}
\DeclareMathOperator{\mad}{mad}
\renewcommand{\leq}{\leqslant}
\renewcommand{\geq}{\geqslant}
\newenvironment{clproof}{\begin{list}{}{%
              \setlength{\leftmargin}{5mm}%
              } \item {\it Proof.} }{\hfill$\lozenge$\end{list}\medskip}

\begin{document}
\title[Defective colouring of graphs excluding a subgraph or minor]{Defective colouring of graphs\\ excluding a subgraph or minor}
\author{Patrice Ossona de Mendez}
\address{Patrice~Ossona~de~Mendez\newline 
Centre d'Analyse et de Math\'ematiques Sociales (CNRS, UMR 8557)\newline
  190-198 avenue de France, 75013 Paris, France\newline 
\hspace*{5mm}	--- and ---\newline 
Computer Science Institute of Charles University (IUUK)\newline
   Malostransk\' e n\' am.25, 11800 Praha 1, Czech Republic}
\email{pom@ehess.fr}
\author{Sang-il Oum}
\address{Sang-il Oum\newline Department of Mathematical Sciences, KAIST\newline Daejeon, South Korea}
\email{sangil@kaist.edu}
\author{David R. Wood}
\address{David R. Wood\newline School of Mathematical Sciences, Monash University\newline  Melbourne, Australia}
\email{david.wood@monash.edu}
\thanks{Ossona de Mendez is supported by grant ERCCZ LL-1201 and by the European Associated Laboratory ``Structures in
Combinatorics'' (LEA STRUCO), and partially supported by ANR project Stint under reference ANR-13-BS02-0007. Research of Wood is supported by the Australian Research Council.}

\date{\today}
\begin{abstract}
Archdeacon (1987) proved that graphs embeddable on a fixed surface can be $3$-coloured so that each colour class induces a subgraph of bounded maximum degree. Edwards, Kang, Kim, Oum and Seymour (2015) proved that graphs with no $K_{t+1}$-minor can be $t$-coloured so that each colour class induces a subgraph of bounded maximum degree. We prove a common generalisation of these theorems with a weaker assumption about excluded subgraphs. This result leads to new defective colouring results for several graph classes, including graphs with linear crossing number, graphs with given thickness (with relevance to the earth--moon problem), graphs with given stack- or queue-number, linklessly or knotlessly embeddable graphs, graphs with given Colin de Verdi\`ere parameter, and graphs excluding a complete bipartite graph as a topological minor. 
\end{abstract}
\maketitle


\section{Introduction}
\label{intro}

A graph $G$ is \emph{$k$-colourable with defect $d$}, or  \emph{$d$-improper $k$-colourable}, or simply \emph{$(k,d)$-colourable}, if each vertex $v$ of $G$ can be assigned one of $k$ colours so that at most $d$ neighbours of $v$ are assigned the same colour as $v$. That is, each monochromatic subgraph has maximum degree at most $d$. Obviously the case $d=0$ corresponds to the usual notion of graph colouring. \citet{CCW1986} introduced the notion of defective graph colouring, and now many results for various graph classes are known. This paper presents $(k,d)$-colourability results for graph classes defined by an excluded subgraph,  subdivision or minor. Our primary focus is on minimising the number of colours $k$ rather than the degree bound $d$. This viewpoint motivates the  following definition. The \emph{defective chromatic number} of a graph class $\mathcal{C}$ is the minimum integer $k$ (if such a $k$ exists) for which there exists an integer $d$ such that every graph in $\mathcal{C}$ is $(k,d)$-colourable. 

Consider the following two examples: \citet{Archdeacon1987} proved that for every surface $\Sigma$, the defective chromatic number of graphs embeddable in $\Sigma$ equals 3. And \citet*{EKKOS2014} proved that the class of graphs containing no $K_{t+1}$ minor has defective chromatic number $t$ (which is a weakening of Hadwiger's conjecture). This paper proves a general theorem that implies both these results as special cases. Indeed, our theorem only assumes an excluded subgraph, which enables it to be applied in more general settings. 

For integers $s,t\geq 1$, let $K_{s,t}^*$ be the bipartite graph obtained from $K_{s,t}$ by adding $\binom{s}{2}$ new vertices, each adjacent to a distinct pair of vertices in the colour class of $s$ vertices in $K_{s,t}$ (see \cref{excluded}). Our main result shows that every graph excluding $K_{s,t}^*$ as a subgraph is $(s,d)$-colourable, where $d$ depends on $s$, $t$ and certain density parameters, which we now introduce. 

\begin{figure}[!t]
        \begin{tikzpicture}
          \tikzstyle{every node}=[circle,draw,fill=black!50,inner sep=0pt,minimum width=4pt]
          \foreach [evaluate=\yy using int(\x+1)]\x in {1,2,3,4,5,6} {
            \foreach \y in {\yy,...,7} {
              \draw (2*\x,0) [bend right] to node [midway] {} (2*\y,0) ;
            }
          }
          \foreach \x in {1,...,7} {
            \draw (2*\x,0) node[fill=red] (a\x){};
          }
          \foreach \y in {1,...,13}  {
             \draw (\y+1,3) node (b\y){};
             \foreach \x in {1,...,7}{
               \draw (a\x) to (b\y);
             }
           }
        \end{tikzpicture}
	\caption{\label{excluded}The graph $K_{7,13}^*$.}
\end{figure}
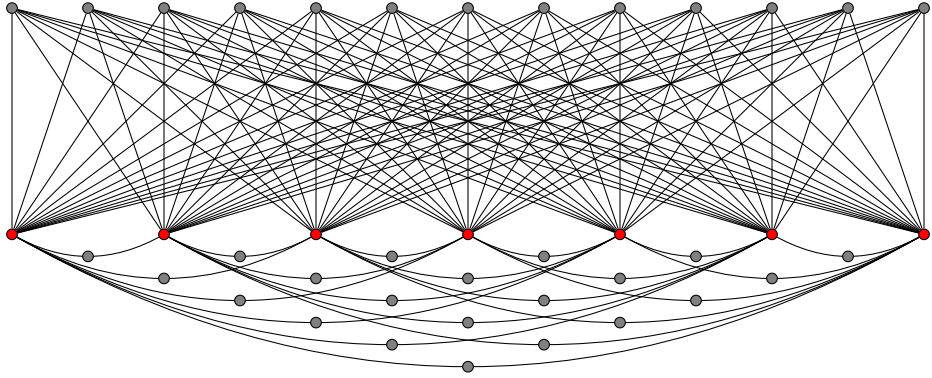

For a graph $G$, the most natural density parameter to consider is the \emph{maximum average degree}, denoted $\mad(G)$, which is the maximum of the average degrees of all subgraphs of $G$; that is, $$\mad(G):=\max_{H\subseteq G} \frac{2|E(H)|}{|V(H)|}.$$ Closely related to maximum average degree is degeneracy. A graph $G$ is \emph{$k$-degenerate} if every subgraph of $G$ has a vertex of degree at most $k$. Every graph is $\lfloor\mad(G)\rfloor$-degenerate. It follows that the chromatic number (and even the choice number) of a graph $G$ is at most $\lfloor\mad(G)+1\rfloor$. Bounds on defective colourings have also been obtained in terms of maximum average degree. In particular, \citet{HS2006} proved that every graph $G$ with $\mad(G)<k+\frac{kd}{k+d}$ is $(k,d)$-colourable, and that there exist non-$(k,d)$-colourable graphs whose maximum average degree tends to $2k$ when $d$ goes to infinity, which shows the limit of the maximum average degree approach for defective colouring (see also \citep{DKMR2014,BK13,BIMR11,MR3528002,MR3300755,MR3095439,MR2876359,MR2489421}).	
	
In addition to maximum average degree we consider the density of shallow topological minors (see \citep{Sparsity} for more on this topic).  A graph $H$ is a \emph{minor} of a graph $G$ if a graph isomorphic to $H$ can be obtained from a subgraph of $G$ by contracting edges. A graph $H$ is a \emph{topological minor} of a graph $G$ if a subdivision of $H$ is a subgraph of $G$. A \emph{$(\leq k)$-subdivision} of a graph $G$ is a graph obtained from $G$ by subdividing each edge at most $k$ times, or equivalently replacing each edge by a path of length at most $k+1$.  The \emph{exact $1$-subdivision} of $G$ is the graph obtained from $G$ by subdividing each edge exactly once. For a half integer $r$ (that is, a number $r$ such that $2r$ is an integer), a graph $H$ is a \emph{depth-$r$ topological minor} of a graph $G$ if a $(\leq 2r)$-subdivision of $H$ is a subgraph of $G$. For a graph $G$, let $G\tnb r$ be the set of all depth-$r$ topological minors of $G$. The \emph{topological greatest reduced average density} (or \emph{top-grad}) with rank $r$ of a graph $G$ is defined as
\[
  \nb[r](G):=\max_{H\in G\tnb r} \frac{|E(H)|}{|V(H)|}.
\]
Note that $\nb[1](G)\geq \nb[0](G)=\frac12 \mad(G)$.

The following is our main result (see \cref{main} for a more precise version).

\begin{THM}
\label{IntroMain}
Every graph $G$ with no $K_{s,t}^*$ subgraph has an $(s,d)$-colouring, where $d$ depends on $s$, $t$, $\mad(G)$ and $\nb(G)$
\end{THM}

We actually prove this result in the setting of defective list colourings, introduced by \citet{Eaton99defectivelist} and since studied by several authors \citep{Chen2014159, Chen2016, choi2016improper, HS2006, LSWZ2001, Skrekovski2000, WX2013, WW2009, Zhang2012332, Zhang2013, Zhang2016}. A \emph{$k$-list assignment} of a graph $G$ is a function $L$ that assigns a set $L(v)$ of exactly $k$ colours to each vertex $v\in V(G)$. Then an \emph{$L$-colouring} of $G$ is a function that assigns a colour in $L(v)$ to each vertex $v\in V(G)$. If an $L$-colouring has the property that every vertex $v$ has at most $d$ neighbours having the same colour as $v$, then we call it an \emph{$L$-colouring with defect $d$}, or \emph{$d$-defective}.  A graph $G$ is \emph{$k$-choosable with defect $d$}, or \emph{$d$-improper $k$-choosable}, or simply \emph{$(k,d)$-choosable} if for every $k$-list assignment $L$ of $G$, there exists a $d$-defective $L$-colouring of $G$. For example, the result of \citet{HS2006} mentioned above applies in the setting of $(k,d)$-choosability. The \emph{defective choice number} of a graph class $\mathcal{C}$ is the minimum integer $k$ (if such a $k$ exists) for which there exists an integer $d$ such that every graph in $\mathcal{C}$ is $(k,d)$-choosable. 

The paper is organised as follows. \cref{MainProof} presents the proof of our main result (a choosability version of \cref{IntroMain}). The subsequent sections present several applications of this main result. In particular, \cref{ExcludedSubgraphs} gives results for graphs with no 4-cycle, and other graph classes defined by an excluded subgraph. \cref{SurfacesCrossings} presents defective 3-colourability results for graphs drawn on surfaces, even allowing for a linear number of crossings, thus generalising the result of Archdeacon mentioned above. \cref{StacksQueuePosets} gives bounds on the defective chromatic number and defective choice number of graphs with given thickness, and of graphs with given stack- or queue-number. One result here is relevant to the earth--moon problem, which asks for the chromatic number of graphs with thickness 2. While it is open whether such graphs are 11-colourable, we prove they are 11-colourable with defect 2. \cref{MinorClosedClasses} studies the defective chromatic number of minor-closed classes. We determine the defective chromatic number and defective choice number of linklessly embeddable graphs, knotlessly embeddable graphs, and graphs with given Colin de Verdi\`ere parameter. We then prove a strengthening of the result of \citet{EKKOS2014} mentioned above. Finally, we formulate a conjecture about the defective chromatic number of $H$-minor-free graphs, and prove several special cases of it. 

\section{Main Proof}
\label{MainProof}

An edge $e$ in a graph is \emph{$\ell$-light} if both endpoints of $e$ have degree at most $\ell$.  There is a large literature on light edges in graphs; see \citep{BS-MN94,JV-DM01,JV02,JenMad96,BSW-DM04,Borodin-JRAM89} for example. Many of our results rely on the following sufficient condition for $(k,d)$-choosability. Its proof is essentially identical to the proof of a lemma by \citet[Lemma 1]{LSWZ2001}. \citet{Skrekovski2000} proved a similar result with $k=1$. 

\begin{LEM}
\label{light}
For integers $\ell\geq k\geq 1$, if every subgraph $H$ of a graph $G$ has a vertex of degree at most $k$ or an $\ell$-light edge, then $G$ is $(k+1,\ell-k)$-choosable.
\end{LEM}

\begin{proof}
Let $L$ be a $(k+1)$-list assignment for $G$. We prove by induction on $\abs{V(H)}+\abs{E(H)}$ that every subgraph $H$ of $G$ is $L$-colourable with defect $\ell-k$. The base case with $\abs{V(H)}+\abs{E(H)}=0$ is trivial. Consider a subgraph $H$ of $G$. If $H$ has a vertex $v$ of degree at most $k$, then by induction $H-v$ is $L$-colourable with defect $\ell-k$, and there is a colour in $L(v)$ used by no neighbour of $v$ which can be assigned to $v$. Now assume that $H$ has minimum degree at least $k+1$. By assumption, $H$ contains an $\ell$-light edge $xy$. By induction, $H-xy$ has an $L$-colouring $c$ with defect $\ell-k$.  If $c(x)\neq c(y)$, then $c$ is also an $L$-colouring of $H$ with defect $\ell-k$. Now assume that $c(x)=c(y)$. We may further assume that $c$ is not an $L$-colouring of $H$ with defect $\ell-k$. Without loss of generality, $x$ has exactly $\ell-k+1$ neighbours (including $y$) coloured by $c(x)$. Since $\deg_H(x)\leq \ell$, there are at most $k-1$ neighbours not coloured by $c(x)$. Since $L(v)$ contains $k$ colours different from $c(x)$, there is a colour used by no neighbour of $x$ which can be assigned to $x$.
\end{proof}

To state our main result, we use the following auxiliary function. For positive integers $s$, $t$ and positive reals $\delta$ and $\delta_1$, let
\begin{equation*}
N_1(s,t,\delta,\delta_1):=
\begin{cases}
(\delta-s)\left(\binom{\lfloor \delta_1\rfloor}{s-1}(t-1)+\tfrac12 \delta_1\right)+\delta &\text{if          }s>2,\\
\tfrac12 (\delta -2)\delta_1t+\delta &\text{if          }s=2,\\
t-1&\text{if }s=1.
\end{cases}
\end{equation*}

\begin{LEM}
\label{findkst}
For positive integers $s$, $t$, and positive reals $\delta$, $\delta_1$, let $\ell=\lfloor N_1(s,t,\delta,\delta_1)\rfloor$. If every subgraph of a graph $G$ has average degree at most $\delta$ and every graph whose exact $1$-subdivision is a subgraph of $G$ has average degree at most $\delta_1$, then  at least one of the following holds:
  \begin{enumerate}[(i)]
  \item $G$ contains a $K^*_{s,t}$ subgraph,
  \item $G$ has a vertex of degree at most $s-1$,
  \item $G$ has an $\ell$-light edge.
  \end{enumerate}
\end{LEM}

\begin{proof}
The case $s=1$ is simple: If (i) does not hold, then $\Delta(G)\leq t-1$, in which case either $G$ has no edges and (ii) holds, or $G$ has an edge and (iii) holds since $\ell=t-1$. Now assume that $s>1$.

Assume for contradiction that $G$ has no $K^*_{s,t}$ subgraph, that every vertex of $G$ has degree at least $s$ (thus $s\leq \delta$), and that $G$ contains no $\ell$-light edge.

Let $A$ be the set of vertices in $G$ of degree at most $\ell$. Let $B:=V(G)\setminus A$.
Let $a:=\abs{A}$ and $b:=\abs{B}$. Since $G$ has a vertex of degree at most $\delta$
and $\delta\leq \ell$, we deduce that $a>0$. Note that no two vertices in $A$ are adjacent.

  Since the average degree of $G$ is at most $\delta$, 
  $$(\ell+1)b+sa\leq 2|E(G)| \leq  \delta(a+b).$$
  That is, 
  \begin{equation}\label{eq:0}
  (\ell+1-\delta)b\leq (\delta-s) a.
  \end{equation}

Let $G'$ be a minor of $G$ obtained from $G-E(G[B])$ by greedily finding a vertex $w\in A$ having a pair of non-adjacent neighbours $x$, $y$ in $B$ and replacing $w$ by an edge joining $x$ and $y$ (by deleting all edges incident with $w$ except $xw$, $yw$ and contracting $xw$), until no such vertex $w$ exists.

  Let $A':=V(G')\setminus B$ and $a':=\abs{A'}$.
  Clearly the exact $1$-subdivision of $G'[B]$ is  a subgraph of $G$. 
  So every subgraph of $G'[B]$ has average degree at most $\delta_1$. 
  Since $G'[B]$ contains at least $a-a'$ edges, 
  \begin{equation}\label{eq:1}
  a-a'\leq \tfrac12 \delta_1 b.
  \end{equation}

A \emph{clique} in a graph is a set of pairwise adjacent vertices.
  Let $M$ be the number of cliques of size $s$ in $G'[B]$.
  Since $G'[B]$ is $\lfloor \delta_1\rfloor$-degenerate, 
  $$M\leq \binom{\lfloor \delta_1\rfloor}{s-1}  b$$ 
  (see \cite[p. 25]{Sparsity} or \citep{Wood2016}).
  If $s=2$, then the following better inequality holds: 
  $$M\leq \tfrac12 \delta_1b.$$

  For each vertex $v\in A'$, since $v$ was not contracted in the creation of $G'$, the set of neighbours of $v$ in $B$ is a clique of size at least $s$. 
  Thus if $a'>  M(t-1)$, 
  then there are at least $t$ vertices in $A'$ sharing at least 
  $s$ common neighbours in $B$.
  These $t$ vertices and their $s$ common neighbours in $B$ 
  with the vertices in $A-A'$ 
  form a $K_{s,t}^*$ subgraph of $G$, contradicting our assumption. 
  Thus, \begin{equation}\label{eq:2}
  a'\leq M(t-1).
  \end{equation}
  By \eqref{eq:0}, \eqref{eq:1} and \eqref{eq:2}, 
$$  \ell+1\leq (\delta-s)\left(\frac{M}{b}(t-1)+\tfrac12 \delta_1\right)+\delta,$$
  contradicting the definition of $\ell$. 
\end{proof}

\cref{light,findkst} imply our main result:

\begin{THM}
\label{main}
  For integers $s,t\geq 1$, every graph $G$ with no $K_{s,t}^*$ subgraph  is $(s,d)$-choosable, where 
  $d:=\lfloor N_1(s,t,\mad(G),2\nb(G))\rfloor -s+1$.
\end{THM}

\begin{proof}
By definition, every subgraph of $G$ has average degree at most $\mad(G)$ and every graph whose exact $1$-subdivision is a subgraph of $G$ has average degree at most $2\nb(G)$. By \cref{findkst}, every subgraph of $G$ has a vertex of degree at most $s-1$ or has an $\ell$-light edge, where $\ell:=\lfloor N_1(s,t,\mad(G),2\nb(G))\rfloor$. By \cref{light} with $k=s-1$, we have that $G$ is $(s,\ell-s+1)$-choosable.
\end{proof}


The following recursive construction was used by \citet{EKKOS2014} to show that their theorem mentioned above is tight. We use this example repeatedly, so include the proof for completeness. If $s=2$, then let $G(s,N):=K_{1,N+1}$. If $s>2$, then let $G(s,N)$ be obtained from the disjoint union of $N+1$ copies of $G(s-1,N)$ by adding one new vertex $v$ adjacent to all other vertices. Note that \citet{HS2006} used a similar construction to prove their lower bound mentioned above.

\begin{LEM}[\citet{EKKOS2014}]
\label{LowerBound}
For integers $s\geq 2$ and $N\geq 1$, the graph $G=G(s,N)$ has no $K_{s,s}$ minor and no $(s-1,N)$-colouring.
\end{LEM}

\begin{proof}
We proceed by induction on $s$. In the base case,  $G=K_{1,N+1}$, and every 1-colouring has a colour class (the whole graph) inducing a subgraph with maximum degree larger than $N$. Thus $G$ is not $(1,N)$-colourable. Now assume that $s\geq 3$ and the claim holds for $s-1$. Let $v$ be the dominant vertex in $G$. Let $C_1,\dots,C_{N+1}$ be the components of $G-v$, where each $C_i$ is isomorphic to $G(s-1,N)$.
  
If $G$ contains a $K_{s,s}$ minor, then some component of $G-v$ contains a $K_{s-1,s-1}$ minor, which contradicts our inductive assumption. Thus  $G$ contains no $K_{s,s}$ minor. 

Suppose that $c$ is an $(s-1)$-colouring of $G(s,N)$. We may assume that $c(v)=1$. If $C_i$ has a vertex of colour $1$ for each $i\in\{1,2,\ldots,N+1\}$, then $v$ has more than $N$ neighbours of colour $1$, which is not possible.   Thus some component $C_i$ has no vertex coloured $1$, and at most $s-2$ colours are used on $C_i$. This contradicts the assumption that $C_i$ has no $(s-2,N)$-colouring.
\end{proof}

Of course, for integers $t\geq s\geq 2$, the graph $G(s,N)$ has no $K_{s,t}$ minor, no $K_{s,t}$ topological minor, and no $K^*_{s,t}$-minor. Thus \cref{LowerBound} shows that the number of colours in \cref{main} is best possible. In other words, \cref{main} states that defective chromatic number and defective choice number of every class of graphs of bounded $\nb$ with no $K_{s,t}^*$ subgraph are at most $s$, and \cref{LowerBound} shows that the number $s$  of colours cannot be decreased.

\section{Excluded Subgraphs}
\label{ExcludedSubgraphs}

This section presents several applications of our main result, in the setting of graph classes defined by an excluded subgraph. Since $K_{s,t}^*$ contains $K_{s,t}$ and a $(\leq 1)$-subdivision of $K_{s+1}$, \cref{main} immediately implies:
\begin{itemize}
\item Every graph $G$ with no $K_{s,t}$ subgraph is $(s,d)$-choosable, where $d$ depends on $s$, $t$, $\mad(G)$ and $\nb(G)$.
\item Every graph $G$ with no subgraph isomorphic to a $(\leq 1)$-subdivision of $K_{s+1}$ is $(s,d)$-choosable, where $d$ depends on $s$, $\mad(G)$ and $\nb(G)$.
\end{itemize}

\subsection{No $4$-Cycles}

Since $K^*_{2,1}$ is the $4$-cycle, \cref{main} with $s=2$ and $t=1$ implies the following.

\begin{COR}
\label{No4cycle}
Every graph $G$ with no $4$-cycle is $(2,d)$-choosable, where $d:=\lfloor 2((\nb[0](G)-1)\nb(G)+\nb[0](G))-1\rfloor$.
\end{COR}

\begin{COR}
Every graph with no $K_t$ minor nor $4$-cycle subgraph  is $(2,O(t^2\log t))$-choosable.
\end{COR}

For planar graphs, $\nb[0]\leq\nb\leq 3$. Indeed $\nb[0]<3$ for planar graphs with at least three vertices. \cref{No4cycle} says that every planar graph with no $4$-cycle is $(2,16)$-choosable. However,  better degree bounds are known.  \citet{BKSY2009} proved that every planar graph with no cycle of length $4$ has a vertex of degree at most $1$ or a $7$-light edge. By \cref{light}, planar graphs with no $4$-cycle are $(2,6)$-choosable. Note that  planar graphs with no $4$-cycle are also known to be $(3,1)$-choosable \citep{WX2013}.

\subsection{Edge Partitions}

\citet{HHLSWZ2002} proved the following theorem on partitioning a graph into edge-disjoint subgraphs.

\begin{THM}[He et al.~{\cite[Theorem 3.1]{HHLSWZ2002}}]
If every subgraph of a graph $G$ has a vertex of degree at most $1$ or an $N$-light edge, then $G$ has an edge-partition into two subgraphs $T$ and $H$ such that $T$ is a forest and $H$ is a graph with maximum degree at most $N-1$.
\end{THM}

This theorem and \cref{findkst} imply the following. 

\begin{THM}
For an integer $t\geq 2$, every graph $G$ with no $K_{2,t}$ subgraph  has an edge-partition into two subgraphs $T$ and $H$ such that   $T$ is a forest and   $H$ has maximum degree at most $\lfloor (\nb[0](G)-1)\nb(G) (t-1) + 2\nb[0](G)\rfloor -1 $.
\end{THM}

\subsection{Nowhere Dense Classes}

A class $\mathcal C$ of graphs is {\em nowhere dense} \citep{ND_logic} if, for every integer $k$ there is some $n$ such that no $(\leq k)$-subdivision of $K_n$ is a subgraph of a graph in $\mathcal C$. Nowhere dense classes are also characterised \citep{ND_characterization} by the property that for every integer $r$ there exists a function $f_r:\mathbb N\rightarrow [0,1]$ with $\lim_{n\rightarrow\infty}f_r(n)=0$ such that every graph $G$ of order $n$ in the class has $\nb[r](G)\leq n^{f_r(n)}$. In other words, for each integer $r$ every graph $G$ in the class has $\nb[r](G)=|V(G)|^{o(1)}$.

For nowhere dense classes, there is no hope to find an improper colouring with a bounded number of colours, since the chromatic number of a nowhere dense class is typically unbounded (as witnessed by the class of graphs $G$ such that $\Delta(G)\leq {\rm girth}(G)$).
However, by the above characterisation, \cref{main} implies there is a partition of the vertex set into a bounded number of parts, each with `small' maximum degree. 

\begin{COR}
Let $\mathcal C$ be a nowhere dense class. Then there exist $c\in\mathbb{N}$ and a function $f:\mathbb{N}\rightarrow [0,1]$ with $\lim_{n\rightarrow\infty} f(n)=0$ such that every $n$-vertex graph in $\mathcal C$ is $(c,n^{f(n)})$-choosable. 
\end{COR}

\section{3-Colouring Graphs on Surfaces}
\label{SurfacesCrossings}

This section considers defective colourings of graphs drawn on a surface, possibly with crossings. First consider the case of no crossings. For example, \citet{CCW1986} proved that every planar graph is $(3,2)$-colourable, improved to $(3,2)$-choosable by \citet{Eaton99defectivelist}. Since $G(3,N)$ is planar, by \cref{LowerBound} the class of planar graphs has defective chromatic-number and defective choice number equal to 3. More generally, \citet{Archdeacon1987} proved the conjecture of \citet{CCW1986} that for every fixed surface $\Sigma$, the class of graphs embeddable in $\Sigma$ has defective chromatic-number 3.  \citet{Woodall11} proved that such graphs have defective choice number 3. It follows from Euler's formula that $K_{3,t}$ is not embeddable on $\Sigma$ for some constant $t$ (see \cref{KstSurface}), and that graphs embeddable in $\Sigma$ have bounded average degree and $\nb$. Thus \cref{main} implies Woodall's result. The lower bound follows from \cref{LowerBound} since $G(3,N)$ is planar. 

\begin{THM}[\citep{Archdeacon1987,Woodall11}] 
\label{Surfaces}
For every surface $\Sigma$, the class of graphs embeddable in $\Sigma$ has defective chromatic-number 3 and defective choice number 3. 
\end{THM}

While our main goal is to bound the number of colours in a defective colouring, we now estimate the degree bound using our method for a graph embeddable in a surface $\Sigma$ of Euler genus $g$. The \emph{Euler genus} of an orientable surface with $h$ handles is $2h$. The \emph{Euler genus} of  a non-orientable surface with $c$
cross-caps is $c$. The \emph{Euler genus} of a graph $G$ is the minimum Euler genus of a surface in which $G$ embeds. For $g\geq 0$, define 
$$d_g:=\max\{3,\tfrac{1}{4}(5+\sqrt{24g+1})\}.$$  
The next two lemmas are well known. We include their proofs for completeness. 

\begin{LEM}
\label{EdgesSurface}
Every $n$-vertex graph $G$ embeddable in a surface of Euler genus $g$ has at most $d_gn$ edges. 
\end{LEM}

\begin{proof}
Suppose that $|E(G)|> d n$, where $d:=d_g$. We may assume that $n\geq 3$.  By Euler's Formula,
$ d n <  |E(G)| \leq 3(n+g-2)$, implying $ (d-3) n < 3g-6$. Since $dn < |E(G)| \leq \binom{n}{2}$ we have $n> 2d + 1$.  Since $d\geq 3$, 
$$ 3g-6 >  (d-3)n \geq (d-3) (2d + 1) =  2d^2 -5d -3.$$ 
Thus 
$2d^2-5d +(3-3g) < 0$. By the quadratic formula, $d 
< \frac14 (5 + \sqrt{1 +24g })$, which is a contradiction. 
Hence $|E(G)|\leq dn$. 
\end{proof}

\begin{LEM}[\citet{Ringel65}]
\label{KstSurface}
For every surface $\Sigma$ of Euler genus $g$, the complete bipartite graph $K_{3,2g+3}$ does not embed in $\Sigma$. 
\end{LEM}

\begin{proof}
By Euler's formula, every triangle-free graph with $n\geq 3$ vertices that embeds in $\Sigma$ has at most $2(n+g-2)$ edges. The result follows. 
\end{proof}

\cref{EdgesSurface,KstSurface,main} imply that graphs embeddable in $\Sigma$ are $(3,O(g^{5/2}))$-choosable. This degree bound is weaker than the  bound of $\max\{15, \frac{1}{2} (3g-8)\}$ obtained by Archdeacon. However, our bound is easily improved. Results by \citet{JT06} and \citet{Ivanco92} show that every graph with Euler genus $g$ has a $(2g+8)$-light edge. Then \cref{light} directly implies that every graph with Euler genus $g$ is $(3,2g+6)$-choosable. Still this bound is weaker than the subsequent improvements to Archdeacon's result of $(3,\max\{12, 6 + \sqrt{6g}\})$-colourability by \citet{CGJ1997} and to $(3,\max\{9, 2 + \sqrt{4g+6}\})$-choosability by \citet{Woodall11}; also see \citep{choi2016improper}. 

\subsection{Linear Crossing Number}

We now generalise \cref{Surfaces} to the setting of graphs with linear crossing number. For an integer $g\geq0$ and real number $k\geq 0$, say a graph $G$ is \emph{$k$-close} to Euler genus $g$ (resp.\ $k$-close to planar) if every subgraph $H$ of $G$ has a drawing on a surface of Euler genus $g$ (resp.\ on the plane) with at most $k\,|E(H)|$ crossings. This says that the average number of crossings per edge is at most $2k$ (for every subgraph). Of course, a graph is planar if and only if it is $0$-close to planar, and a graph has Euler genus at most $g$ if and only if it is $0$-close to Euler genus $g$. Graphs that can be drawn in the plane with at most $k$ crossings per edge, so called \emph{$k$-planar graphs}, are examples of graphs $(\frac{k}{2})$-close to planar. \citet{PachToth97} proved that $k$-planar graphs have average degree $O(\sqrt{k})$. It follows that $k$-planar graphs are $O(\sqrt{k})$-colourable, which is best possible since $K_n$ is $O(n^2)$-planar. For defective colourings, three colours suffice even in the more general setting of graphs $k$-close to Euler genus $g$. 

\begin{THM}
\label{typegk}
For all integers $g,k\geq0$ the class of graphs $k$-close to Euler genus $g$ has defective chromatic number and defective choice number equal to 3. In particular, every graph  $k$-close to Euler genus $g$ is $(3,O((k+1)^{5/2}(g+1)^{7/2}))$-choosable.
\end{THM}

We prove this theorem by a series of lemmas, starting with  a straightforward extension of the standard probabilistic proof of the crossing lemma. Note that  \citet{SSSV96} obtained a better bound for a restricted range of values for $m$ relative to $n$.

\begin{LEM}
\label{GeneralisedCrossingLemma}
Every drawing of a graph with $n$ vertices and $m\geq 2d_gn$ edges on a surface $\Sigma$ of Euler genus $g$ has at least $m^3/(8(d_gn)^2)$ crossings. 
\end{LEM}

\begin{proof}
By \cref{EdgesSurface}, every $n$-vertex graph that embeds in $\Sigma$ has at most $d_gn$ edges. Thus every drawing of an $n$-vertex $m$-edge graph on $\Sigma$ has at least $m-d_gn$ crossings. 

Given a graph $G$ with $n$ vertices and $m\geq 2d_gn$ edges and a crossing-minimal drawing of $G$ on $\Sigma$, choose each vertex of $G$ independently and randomly with probability $p:=2d_gn/m$. Note that $p\leq 1$. Let $G'$ be the induced subgraph obtained. The expected number of vertices in $G'$ is $pn$, the expected number of edges in $G'$ is $p^2m$, and the expected number of crossings in the induced drawing of $G'$ is $p^4c$, where $c$ is the number of crossings in the drawing of $G$. By linearity of expectation and the above naive bound, $p^4 c \geq p^2m-d_g\, pn$. 
Thus $c \geq (pm- d_g n)/p^3 = d_gn/p^3 =  m^3/ (8(d_gn)^2)$. 
\end{proof} 

This lemma leads to the following bound on the number of edges.

\begin{LEM}
\label{gkEdges}
If an $n$-vertex $m$-edge graph $G$ has a drawing on a surface of Euler genus $g$ with at most $km$ crossings, then 
$$m\leq \sqrt{8k+4}\,d_gn.$$
\end{LEM}

\begin{proof}
If $m< 2d_gn$ then $m<\sqrt{8k+4}\,d_gn$, and we are done. 
Otherwise, $m\geq 2d_gn$, and \cref{GeneralisedCrossingLemma} is applicable. 
Thus every drawing of $G$ on a surface of Euler genus $g$ has at least $m^3/(8(d_gn)^2)$ crossings. 
Hence $m^3/(8(d_gn)^2)\leq km$, implying $m \leq \sqrt{8k}\,d_gn$. 
\end{proof}

To apply \cref{main} we bound the size of $K_{3,t}$ subgraphs.

\begin{LEM}
\label{K3t}
Every drawing of $K_{3,t}$ in a surface of Euler genus $g$ has at least 
$$ \frac{t(t-1)}{(2g+3)(2g+2)}$$
crossings.
\end{LEM}

\begin{proof}
By \cref{KstSurface}, $K_{3,2g+3}$ does not embed (crossing-free) in a surface of Euler genus $g$. 
Consider a drawing of $K_{3,t}$ in a surface of Euler genus $g$.  
There are $\binom{t}{2g+3}$ copies of $K_{3,2g+3}$ in $K_{3,t}$. 
Each such copy has a crossing. 
Each crossing is in at most $\binom{t-2}{2g+1}$ copies of $K_{3,2g+3}$. 
Thus the number of crossings is at least 
$$\binom{t}{2g+3} \bigg/ \binom{t-2}{2g+1}
= \frac{t(t-1)}{(2g+3)(2g+2)}.$$
\end{proof}

\begin{LEM}
\label{typegkK3t}
If a graph $G$ is $k$-close to Euler genus $g$ and contains $K_{3,t}$ as a subgraph, then 
$$t \leq 3k(2g+3)(2g+2)+1.$$
\end{LEM}

\begin{proof}
Suppose that $G$ contains $K_{3,t}$ as a subgraph.
Since $G$ is $k$-close to Euler genus $g$, so is $K_{3,t}$. 
Thus $K_{3,t}$ has a drawing in a surface of Euler genus $g$ where the number of crossings is at most $3kt$.
By \cref{K3t},
$$ \frac{t(t-1)}{(2g+3)(2g+2)} \leq 3kt. $$
The result follows. 
\end{proof}

We now prove the main result of this section. 

\begin{proof}[Proof of \cref{typegk}] 
Say $G$ is a  graph $k$-close to Euler genus $g$. 
By \cref{typegkK3t},  $G$ contains no $K_{3,t}$ with $t=3k(2g+3)(2g+2)+2$. 
By \cref{gkEdges}, $\mad(G)\leq 2\sqrt{8k+4}\,d_g$. 
We now bound $\nb(G)$. 
Consider a subgraph $H$ of $G$ that is a $(\leq 1$)-subdivision of a graph $X$. 
Since $G$ is $k$-close to Euler genus $g$, so is $H$. 
Thus $H$ has a drawing on a surface of Euler genus $g$ with at most $k|E(H)|$ crossings. 
Remove each division vertex and replace its two incident edges by one edge.
We obtain a drawing of $X$ with the same number of crossings as the drawing of $H$. 
Now $|E(H)|\leq 2|E(X)|$. 
Thus $X$ has a drawing on a surface of Euler genus $g$ with at most $2k|E(X)|$ crossings. 
By \cref{gkEdges}, 
$|E(X)|\leq \sqrt{16k+4}\,d_g|V(X)|$. 
Hence $\nb(G)\leq \sqrt{16k+4}\,d_g$. 
By \cref{main}, $G$ is $(3,d)$-choosable, where
\begin{align*}
d&=\lfloor N_1(3,3k(2g+3)(2g+2)+2,2\sqrt{8k+4}\,d_g,2\sqrt{16k+4}\,d_g)\rfloor-2\\
& \leq O((k+1)^{5/2}(g+1)^{7/2}).\qedhere
\end{align*}
\end{proof}

\section{Thickness Parameters}
\label{StacksQueuePosets}

This section studies defective colourings of graphs with given thickness or other related parameters. 
\citet{Yancey12} first proposed studying defective colourings of graphs with given thickness. 

\subsection{Light Edge Lemma}

Our starting point is the following sufficient condition for a graph to have a light edge. The proof uses a technique by \citet{BSW-DM04}, which we present in a general form. 

\begin{LEM} 
\label{LightEdgeGen}
Let $G$ be a graph with $n$ vertices, at most $an+b$ edges, and minimum degree $\delta$, such that every spanning bipartite subgraph has at most $a'n+b'$ edges, for some $a,a'\in\mathbb{R}^+$ and $b,b'\in\mathbb{R}$ and $\delta\in\mathbb{Z}^+$ satisfying:
\begin{align}
& 2a \geq \delta >a', \label{LightEdgeConditionA}\\
& (\delta-a')\ell  >(2a-a')\delta, \text{ and} \label{LightEdgeConditionB}\\
& (\delta-a')\ell^2 -  \big( (2a-a')\delta+b'-\delta+a'\big)  \ell - (2a-a'+2b-b')\delta  >0. \label{LightEdgeCondition}
\end{align}
Then $G$ has an $(\ell-1)$-light edge.
\end{LEM}
\begin{proof}
Let $X$ be the set of vertices with degree at most $\ell-1$. 
Since vertices in $X$ have degree at least $\delta$ and vertices not in $X$ have degree at least $\ell$,
$$\delta |X| + (n-|X|)\ell \leq \sum_{v\in V(G)} \deg(v)  = 2|E(G)| \leq 2(an+b) .$$ 
Thus 
\begin{equation*}
(\ell-2a)n -2b \leq (\ell-\delta)|X|.
\end{equation*}
Suppose on the contrary that $X$ is an independent set in $G$. Let $G'$ be the spanning bipartite subgraph of $G$ consisting of all edges between $X$ and $V(G)\setminus X$. Since each of the at least $\delta$ edges incident with each vertex in $X$ are in $G'$, 
$$\delta  |X| \leq |E(G')| \leq a'n+b'.$$ 
Since $\ell>\frac{2a-a'}{\delta-a'}\delta>\delta$ (hence  $\ell-\delta > 0$)  and $\delta\geq 0$, 
\begin{align*}
& & \delta (\ell-2a)n-2b\delta  & \leq \delta (\ell-\delta)|X| \leq (\ell-\delta)(a'n+b')\\
& \Rightarrow & \big( \delta (\ell-2a) - a'(\ell-\delta)\big) n  & \leq (\ell-\delta)b' + 2b\delta\\
& \Rightarrow & \big( (\delta-a')\ell -(2a-a')\delta\big) n  & \leq b'\ell +(2b-b')\delta.
\end{align*}
If $n\leq \ell$ then every edge is $(\ell-1)$-light. Now assume that $n\geq \ell+1$. 
Since $(\delta-a')\ell -(2a-a')\delta>0$,
$$
 \big( (\delta-a')\ell -(2a-a')\delta\big) (\ell+1)   \leq b'\ell +(2b-b')\delta. 
$$
Thus
$$
(\delta-a')\ell^2+ \big( \delta-a' -(2a-a')\delta-b'\big)  \ell  \leq  (2a-a'+2b-b')\delta, 
$$
which is a contradiction. Thus $X$ is not an independent set. Hence $G$ contains an $(\ell-1)$-light edge. 
\end{proof}

\begin{REM}
To verify \eqref{LightEdgeCondition}, the following approximation can be useful:
If $\alpha,\beta,\gamma$ are strictly positive reals, then the larger root of $\alpha x^2-\beta x-\gamma=0$ is at most
\begin{equation}
\label{RootApprox}
	\frac{\beta+\sqrt{\beta(\beta+\frac{4\alpha\gamma}{\beta})}}{2\alpha}\leq \frac{\beta+\frac{1}{2}(\beta+(\beta+\frac{4\alpha\gamma}{\beta}))}{2\alpha}=\frac{\beta}{\alpha}+\frac{\gamma}{\beta}.
\end{equation}
\end{REM}

%

\cref{light} with $k=\delta-1$ and \cref{LightEdgeGen} imply the following sufficient condition for defective choosability. With $\delta:=\floor{a'}+1$, which is the minimum possible value for $\delta$, the number of colours only depends on the coefficient of $|V(H)|$ in the bound on the number of edges in a bipartite subgraph $H$. 

\begin{LEM} 
\label{LightEdgeColour}
Fix constants $a,a'\in \mathbb R^+$ and $b,b'\in \mathbb R $ and  $\ell,\delta\in \mathbb Z^+$ satisfying 
\eqref{LightEdgeConditionA}, \eqref{LightEdgeConditionB} and \eqref{LightEdgeCondition}. Let $G$ be a graph such that every subgraph $H$ of $G$ with  minimum degree at least $\delta$ satisfies the following conditions:
\begin{enumerate}[(i)]
\item $H$ has at most $a|V(H)|+b$ edges.
\item Every spanning bipartite subgraph of $H$ has at most  $a'|V(H)|+b'$ edges.
\end{enumerate}
Then $G$ is $(\delta,\ell-\delta)$-choosable. In particular, $G$ is $(\floor{a'}+1,\ell-1-\floor{a'})$-choosable.
\end{LEM}

%

\cref{LightEdgeGen} with $a=3$ and $b=3(g-2)$ and $a'=2$ and $b'=2(g-2)$ and $\ell=2g+13$ implies that every graph $G$ with minimum degree at least $3$ and Euler genus $g$ has a $(2g+12)$-light edge. Note that this bound is within $+10$ of being tight since $K_{3,2g+2}$ has minimum degree 3, embeds in a surface of Euler genus $g$, and every edge has an endpoint of degree $2g+2$. More precise results, which are typically proved by discharging with respect to an embedding, are known \citep{JT06,Ivanco92,Borodin-JRAM89}. \cref{LightEdgeColour} then implies that every graph with Euler genus $g$ is $(3,2g+10)$-choosable. As mentioned earlier, this result with a better degree bound was proved by \citet{Woodall11}; also see \citep{choi2016improper}. The utility of \cref{LightEdgeColour} is that it is immediately applicable in more general settings, as we now show. 

\subsection{Thickness} 

The \emph{thickness} of a graph $G$ is the minimum integer $k$ such that $G$ is the union of $k$ planar subgraphs; see  \citep{MOS98} for a survey on thickness. A minimum-degree-greedy algorithm properly $6k$-colours a graph with thickness $k$, and it is an open problem to improve this bound for $k\geq 2$. The result of \citet{HS2006} implies that graphs with thickness $k$, which have maximum average degree less than $6k$, are $(3k+1,O(k^2))$-choosable, but gives no result with at most $3k$ colours. We show below that graphs with thickness $k$ are $(2k+1,O(k^2))$-choosable, and that no result with at most $2k$ colours is possible. That is, both the defective chromatic number and defective choice number of the class of graphs of thickness at most $k$ equal $2k+1$. In fact, the proof works in the following more general setting. For an integer $g\geq 0$, the \emph{$g$-thickness} of a graph $G$ is the minimum integer $k$ such that $G$ is the union of $k$ subgraphs each with Euler genus at most $g$. This definition was implicitly introduced by \citet{JR00}. By Euler's Formula, every graph with $n\geq 3$ vertices and $g$-thickness $k$ has at most $3k(n+g-2)$ edges, and every spanning bipartite subgraph has at most $2k(n+g-2)$ edges. \cref{LightEdgeGen} with $\ell=2kg+8k^2+4k+1$ (using \eqref{RootApprox} to verify \eqref{LightEdgeCondition}) implies:

\begin{LEM} 
\label{LightEdgeGenusThickness}
Every graph with minimum degree at least $2k+1$ and $g$-thickness at most $k$ has a $(2kg+8k^2+4k)$-light edge.
\end{LEM}

We now determine the defective chromatic number and defective choice number of graphs with given $g$-thickness. 

\begin{THM} 
\label{ColourGenusThickness}
For integers $g\geq 0$ and $k\geq 1$, the class of graphs with $g$-thickness at most $k$ has defective chromatic number and defective choice number equal to $2k+1$. In particular, every graph with $g$-thickness at most $k$ is $(2k+1,2kg+8k^2+2k)$-choosable.
\end{THM}

\begin{proof}
\cref{LightEdgeColour,LightEdgeGenusThickness} imply the upper bound. 
As usual, the lower bound is provided by $G(2k+1,N)$. 
We now prove that $G=G(2k+1,N)$ has $g$-thickness at most $k$  by induction on $k$ (with $g$ fixed). 
Note that $G(3,N)$ is planar, and thus has $g$-thickness 1. 
Let $r$ be the vertex of $G$ such that $G-r$ is the disjoint union of $N+1$ copies of $G(2k,N)$. 
For $i\in[N+1]$, let $v_i$ be the vertex of the $i$-th component $C_i$ of $G-r$ such that $C_i-v_i$ is the disjoint union of $N+1$ copies of $G(2k-1,N)$. Let $H:=G-\{r,v_1,v_2,\ldots,v_{N+1}\}$.
Observe that each component of $H$ is isomorphic to $G(2k-1,N)$ and by induction, $H$ has $g$-thickness at most $k-1$. Since $G-E(H)$ consists of $N+1$ copies of $K_{2,N'}$ pasted on $r$ for some $N'$, $G-E(H)$ is planar and thus has $g$-thickness 1.  
Hence $G$ has $g$-thickness at most $k$. By \cref{LowerBound}, $G(2k+1,N)$ has no $(2k,N)$-colouring. 
Therefore the class of graphs with $g$-thickness at most $k$ has defective chromatic number and defective choice number at least $2k+1$. 
\end{proof}

The case $g=0$ and $k=2$ relates to the famous earth--moon problem \citep{ABG11,GS09,Hut93,Ringel59,JR00}, which asks for the maximum chromatic number of graphs with thickness $2$. The answer is in $\{9,10,11,12\}$. The result of \citet{HS2006} mentioned in \cref{intro} implies that graphs with thickness 2 are $(7,18)$-choosable, $(8,9)$-choosable, $(9,5)$-choosable, $(10,3)$-choosable, and $(11,2)$-choosable because their maximum average degree is less than $12$. But their result gives no bound with at most 6 colours. \cref{ColourGenusThickness} says that the class of graphs with thickness 2 has defective chromatic number and defective choice number equal to $5$. In particular, \cref{LightEdgeColour} implies that graphs with thickness 2 are $(5,36)$-choosable, $(6,19)$-choosable, $(7,12)$-choosable, $(8,9)$-choosable, $(9,6)$-choosable, $(10,4)$-choosable, and $(11,2)$-choosable. This final result, which is also implied by the result of \citet{HS2006}, is very close to the conjecture that graphs with thickness 2 are 11-colourable. Improving these degree bounds provides an approach for attacking the earth--moon problem.

%
%
%
%
%
%
%
%
%
%
%


\subsection{Stack Layouts}
A \emph{$k$-stack layout} of a graph $G$ consists of a linear ordering $v_1,\dots,v_n$ of $V(G)$ and a partition $E_1,\dots,E_k$ of $E(G)$ such that no two edges in $E_i$ cross with respect to $v_1,\dots,v_n$ for each $i\in[1,k]$. Here edges $v_av_b$ and $v_cv_d$  \emph{cross} if $a<c<b<d$. A graph is a \emph{$k$-stack graph} if it has a $k$-stack layout. The \emph{stack-number} of a graph $G$ is the minimum integer $k$ for which $G$ is a $k$-stack graph. Stack layouts are also called \emph{book embeddings}, and stack-number is also called \emph{book-thickness}, \emph{fixed outer-thickness} and \emph{page-number}. The maximum chromatic number of $k$-stack graphs is in $\{2k,2k+1,2k+2\}$; see  \citep{DujWoo04}. For defective colourings, $k+1$ colours suffice.  
{\interlinepenalty1000\par} 

\begin{THM}
\label{stack}
The class of $k$-stack graphs has defective chromatic number and defective choice number equal to $k+1$. In particular, every $k$-stack graph is $(k+1,2^{O(k\log k)})$-choosable. 
\end{THM}

\begin{proof}
The lower bound follows from \cref{LowerBound} since an easy inductive argument shows that $G(k+1,N)$ is a $k$-stack graph for all $N$. For the upper bound, $K_{k+1,k(k+1)+1}$ is not a $k$-stack graph \citep{BERNHART1979320}; see also \citep{deKlerk201480}. Every $k$-stack graph $G$ has average degree less than $2k+2$ (see \citep{Keys75,BERNHART1979320,DujWoo04}) and  $\nb(G)\leq 20k^2$ (see \citep{NOW11}). The result follows from  \cref{main} with $s=k+1$ and $t=k(k+1)+1$, where $\lfloor N_1(k+1,k(k+1)+1,2k+2,40k^2)\rfloor -k\leq 2^{O(k\log k)}$.  
\end{proof}

\subsection{Queue Layouts}

A \emph{$k$-queue layout} of a graph $G$ consists of a linear ordering $v_1,\dots,v_n$ of $V(G)$ and a partition $E_1,\dots,E_k$ of $E(G)$ such that no two edges in $E_i$ are nested with respect to $v_1,\dots,v_n$ for each $i\in[1,k]$. Here edges $v_av_b$ and $v_cv_d$ are \emph{nested} if $a<c<d<b$. The \emph{queue-number} of a graph $G$ is the minimum integer $k$ for which $G$ has a $k$-queue layout. A graph is a \emph{$k$-queue graph} if it has a $k$-queue layout. \citet{DujWoo04} state that determining the maximum chromatic number of $k$-queue graphs is an open problem, and showed lower and upper bounds of  $2k+1$ and $4k$. We provide the following partial answer to this question. 

\begin{THM}\label{queue}
Every $k$-queue graph is $(2k+1,2^{O(k\log k)})$-choosable.
\end{THM}

\begin{proof}
\citet{Heath} proved that $K_{2k+1,2k+1}$ is not a $k$-queue graph. Every $k$-queue graph $G$ has $\mad(G)<4k$ (see \citep{Heath,Pemmaraju-PhD,DujWoo04}) and $\nb(G)<(2k+2)^2$ (see \citep{NOW11}). The result then follows from \cref{main} with $s=2k+1$ and $t=2k+1$, where 
$\lfloor N_1(2k+1,2k+1,4k,2(2k+2)^2)\rfloor - 2k\leq 2^{O(k\log k)}$.
\end{proof}

Since $G(k+1,n)$ has a $k$-queue layout, the defective chromatic number of the class of $k$-queue graphs is at least $k+1$ and at most $2k+1$ by \cref{LowerBound} and \cref{queue}. It remains an open problem to determine its defective chromatic number.


%
%

\subsection{Posets}

Consider the  problem of partitioning the domain $X$ of a given poset $P=(X,\preceq)$ into $X_1,\dots,X_k$ so that each $(X_i,\preceq)$ has small poset dimension. The \emph{Hasse diagram} $H(P)$ of $P$ is the graph whose vertices are the elements of $P$ and whose edges correspond to the \emph{cover relation} of $P$. Here $x$ \emph{covers} $y$ in $P$ if $y\neq x$, $y \preceq x$ and there is no element $z$ of $P$ such that $z\neq y$, $z\neq x$, and $y\preceq z\preceq x$.  A \emph{linear extension} of $P=(X,\preceq)$ is a total order $\le$ on $X$ such that $x\preceq y$ implies $x\le y$ for every $x,y\in X$. The \emph{jump number} of $P$ is the minimum number of consecutive elements of a linear extension of $P$ that are not comparable in $P$, where the minimum is taken over all possible linear extensions of $P$.

\begin{THM}
For every integer $k$ there is an integer $d$ such that the domain of any poset  with jump number at most $k$ can be coloured with $2k+3$ colours, such that each colour induces a poset with dimension at most $d$.
\end{THM}

\begin{proof}
\citet{HP-SJDM97} showed that the queue-number of the Hasse diagram of a poset $P$ is at most one more than the jump number of $P$, and \citet{Furedi1986} proved that if the Hasse diagram of a poset has maximum degree $\Delta$, then its dimension is at most $50\Delta(\log\Delta)^2$. The result then follows from \cref{queue}.
\end{proof}


\section{Minor-Closed Classes}
\label{MinorClosedClasses}

This section shows that for many minor-closed classes, \cref{main} determines the defective chromatic number and defective choice number. For example, every outerplanar graph has average degree less than $4$ and contains no $K_{2,3}$ subgraph. Thus \cref{main} implies that every outerplanar graph is $(2,14)$-choosable. A better degree bound was obtained by \citet{CCW1986}, who proved that outerplanar graphs are $(2,2)$-colourable. Since $G(1,N)$ is outerplanar, by \cref{LowerBound} the defective chromatic number and defective choice number of the class of outerplanar graphs equal 2. As shown in \cref{SurfacesCrossings}, the  defective chromatic number and defective choice number of the class of graphs embeddable in any fixed surface equal 3. We now consider some other minor-closed classes. 

\subsection{Linklessly and Knotlessly Embeddable Graphs}

A graph is \emph{linklessly embeddable} if it has an embedding in $\mathbb{R}^3$ with no two topologically
 linked cycles  \citep{Sachs83,RST93a}. Linklessly embeddable graphs form a minor-closed class whose minimal excluded minors are the so-called Petersen family \citep{RST95}, which includes $K_6$, $K_{4,4}$ minus an edge, and the Petersen graph. Since linklessly embeddable graphs exclude $K_6$ minors, they are $5$-colourable \citep{RST93} and $8$-choosable \citep{BJW11}. It is open whether $K_6$-minor-free graphs or linklessly embeddable graphs are $6$-choosable. A graph is \emph{apex} if deleting at most one vertex makes it planar. Every apex graph is linklessly embeddable \citep{RST93a}. Since $G(3,N)$ is planar, $G(4,N)$ is apex, and thus linklessly embeddable. By \cref{LowerBound}, the class of linklessly embeddable graphs has defective chromatic number at least $4$. Mader's theorem \cite{Mader68} for $K_6$-minor-free graphs implies that linklessly embeddable graphs have average degree less than 8 and minimum degree at most 7. Since linklessly embeddable graphs exclude $K_{4,4}$ minors, \cref{main} implies the following result.

\begin{THM}
\label{Linkless}
The class of linklessly embeddable graphs has defective chromatic number and defective choice number $4$. In particular, every linklessly embeddable graph is $(4,440)$-choosable. 
\end{THM}



A graph is \emph{knotlessly embeddable} if it has an embedding in $\mathbb{R}^3$ in which every cycle forms a trivial knot; see \citep{Alfonsin05} for a survey. Knotlessly embeddable graphs form a minor-closed class whose minimal excluded minors include $K_7$ and $K_{3,3,1,1}$ \citep{CG83,MR1883594}. 
More than 260 minimal excluded minors are known~\cite{MR3212585}, but the full list of minimal excluded minors is unknown. Since knotlessly embeddable graphs exclude $K_7$ minors, they are $8$-colourable \citep{AG13,Jakobsen71}. \citet{Mader68} proved that $K_7$-minor-free graphs have average degree less than 10, which implies they are $9$-degenerate and thus $10$-choosable. It is open whether $K_7$-minor-free graphs or knotlessly embeddable graphs are $6$-colourable or $7$-choosable \citep{BJW11}. A graph is \emph{$2$-apex} if deleting at most two vertices makes it planar. \citet{MR2341314} and \citet{MR2351115} proved that every $2$-apex graph is knotlessly embeddable. Since every block of $G(5,N)$ is $2$-apex, $G(5,N)$ is knotlessly embeddable. By \cref{LowerBound}, the class of knotlessly embeddable graphs has defective chromatic number at least $5$. Since $K_{3,3,1,1}$ is a minor of $K_{5,3}^*$, knotlessly embeddable graphs do not contain a $K_{5,3}^*$ subgraph. Since knotlessly embeddable graphs have average degree less than $10$, \cref{main} implies the following result.

\begin{THM}
The class of knotlessly embeddable graphs has defective chromatic number and defective choice number $5$. In particular, every knotlessly embeddable graph is $(5,660)$-choosable. 
\end{THM}

\subsection{Excluded Complete and Complete Bipartite Minors}

Now consider graphs excluding a given complete graph as a minor. \citet{EKKOS2014} proved that the class of $K_{s+1}$-minor-free graphs has defective chromatic-number $s$, which is a weakening of Hadwiger's conjecture. They also noted that the same method proves the same result for $K_{s+1}$-topological minor-free graphs. We have the following choosability versions of these results.

\begin{THM}
\label{Ktminor}
For each integer $s\geq 2$, the class of $K_{s+1}$-minor-free graphs has defective chromatic-number $s$ and defective choice number $s$. In particular, if $\delta$ is the maximal density of a $K_{s+1}$-minor-free graph, then every  $K_{s+1}$-minor-free graph is 
$(s,\lfloor \delta(2\delta-s+1)\rfloor-s+1)$-choosable. The same result holds replacing ``minor'' by ``topological minor''. 
\end{THM}

The lower bound in \cref{Ktminor} follows from \cref{LowerBound}. The upper bound follows from \cref{main} with $t=1$ since $K^*_{s,1}$ has a $K_{s+1}$-topological-minor. Indeed, in the $t=1$ case, the proof of \cref{main} is the same as the proof of \citet{EKKOS2014} with essentially the same degree bound. For $K_{s+1}$-minor-free graphs, \citet{Kostochka82,Kostochka84} and \citet{Thomason84,Thomason01} proved that the maximum density $\delta=\Theta(s\sqrt{\log s})$, and thus every  $K_{s+1}$-minor-free graph is $(s,O( s^2\log s))$-choosable. For $K_{s+1}$-topological-minor-free graphs,  \citet{BT1998} and \citet{KS1996b} proved that the maximum density $\delta=\Theta(s^2)$, and thus every  $K_{s+1}$-topological-minor-free graph is $(s,O(s^4))$-choosable. Finally, note that for $K_{s+1}$-minor-free graphs, choice number and defective choice number substantially differ, since \citet{BJW11} constructed $K_{s+1}$-minor-free graphs that are not $\frac{4}{3}(s-1)$-choosable (for infinitely many $s$).

Now we deduce a theorem for the class of graphs with no $K_{s,t}$ topological minor. 

\begin{THM}
\label{KstTopoMinor}
For integers $t\geq s\geq 1$, the defective chromatic number and the defective choice number of the class of $K_{s,t}$ topological minor-free graphs are equal to $s$. In particular, every $K_{s,t}$ topological minor-free graph is $(s,2^{O(s\log t)} )$-choosable. 
\end{THM}

\begin{proof}
The lower bound follows from \cref{LowerBound}, since $G(s,N)$ contains no $K_{s,t}$ topological minor. For the upper bound, 
\citet{ReedWood16} noted that a method of \citet{Diestel4}, which is based on  a result about linkages due to \citet{TW05}, shows that for every graph $H$ with $p$ vertices and $q$ edges, every graph with average degree at least $4p + 20q$ contains $H$ as a topological minor. Thus every $K_{s,t}$-topological-minor-free graph $G$ has $\mad(G) \leq 20st+4(s+t) \leq 4(5s+2)t$ and  $\nb(G)\leq 2(5s+2)t$. 
By  \cref{main}, $G$ is $(s,d)$-choosable, where $d:=\lfloor N_1(s,t,4(5s+2)t,4(5s+2)t)-s+1 \rfloor$, which is in $2^{O(s\log t)}$. 
\end{proof}

Note that \cref{KstTopoMinor} implies and is more general than \cref{Ktminor}, since $K_{s,t}$ contains $K_{s+1}$ as a minor (for $t\geq s$). For $K_{s,t}$-minor-free graphs, the degree bound in \cref{KstTopoMinor} can be improved by using known results on the extremal  function for $K_{s,t}$-minor-free graphs \citep{KP12,KP10,KP08,KO05,KO-Comb05,HarveyWood16}.


\subsection{Colin de Verdi\`ere Parameter}

The  Colin de Verdi\`ere parameter $\mu(G)$ is an important graph invariant introduced by \citet{CdV90,CdV93}; see \citep{Holst97,HLS,Schrijver97} for surveys. It is known that $\mu(G)\leq 1$ if and only if $G$ is a forest of paths, $\mu(G)\leq 2$ if and only if $G$ is outerplanar, $\mu(G)\leq 3$ if and only if $G$ is planar, and 
$\mu(G)\leq 4$ if and only if $G$ is linklessly embeddable. A famous conjecture of \citet{CdV90} states that $\chi(G)\leq \mu(G)+1$ (which implies the 4-colour theorem, and is implied by Hadwiger's Conjecture). For defective colourings one fewer colour suffices. 

\begin{THM}
\label{Colin}
For $k\geq 1$, the defective chromatic number and the defective choice number of the class of graphs $G$ with $\mu(G)\leq k$  are equal to $k$. In particular, every graph $G$ with $\mu(G)\leq k$ is $(k,2^{O(k\log\log k)})$-choosable. 
\end{THM}

\begin{proof}
Graphs with $\mu(G)\leq k$ form a minor-closed class \citep{CdV90,CdV93}. \citet{HLS} proved that $\mu(K_{s,t}) = s+1$ for $t\geq\max\{s,3\}$. Thus, if $\mu(G)\leq k$ then $G$ contains no $K_{k,\max(k,3)}$ minor, and $\mad(G)\leq 2\nb(G)\leq  O(k\sqrt{\log k})$.  \cref{main} with $s=k$ and $t=\max\{k,3\}$ implies that $G$ is $(k,2^{O(k\log\log k)})$-choosable. Now we prove the lower bound. \citet{HLS} proved that $\mu(G)$ equals the maximum of $\mu(G')$, taken over the components $G'$ of $G$, and if $G$ has a dominant vertex $v$, then $\mu(G)=\mu(G-v)+1$. It follows that $\mu(G(k,N))=k$ for $N\geq 2$.  \cref{LowerBound} then implies that the class of graphs with $\mu(G)\leq k$ has defective chromatic number and defective choice number at least $k$. 
\end{proof}

 \cref{Colin}  generalises \cref{Linkless} which corresponds to the case $k=4$.

\subsection{$H$-Minor-Free Graphs}

This section considers, for an arbitrary graph $H$, the defective chromatic number of the class of $H$-minor-free graphs, which we denote by $f(H)$. That is, $f(H)$ is the minimum integer such that there exists an integer $d(H)$ such that every $H$-minor-free graph has a $(f(H),d(H))$-colouring. Obviously, $f$ is  minor-monotone: if $H'$ is a minor of $H$, then every $H'$-minor-free graph is $H$-minor-free, and thus $f(H')\leq f(H)$. 

A set $S$ of vertices in a graph $H$ is a \emph{vertex cover} if $E(H-S)=\emptyset$. Let $\tau(H)$ be the minimum size of a vertex cover in $H$, called the \emph{vertex cover number} of $H$. The \emph{tree-depth} of a connected graph $H$, denoted by $\td(H)$, is the minimum height of a rooted tree $T$ such that $H$ is a subgraph of the closure of $T$. Here the \emph{closure} of $T$ is obtained from $T$ by adding an edge between every ancestor and descendent in $T$. The \emph{height} of a rooted tree is the maximum number of vertices on a root--to--leaf path. The \emph{tree-depth} of a disconnected graph $H$ is the maximum tree-depth of the connected components of $H$.

\begin{PROP}
\label{Hfree}
	For every graph $H$,
	\begin{equation*}
		{\rm td}(H)-1\leq f(H)\leq \tau(H).
	\end{equation*}
\end{PROP}

\begin{proof}
Obviously, $H$ is a minor of $K^*_{\tau(H), |V(H)| - \tau(H)}$. 
Thus every $H$-minor-free graph is $K^*_{\tau(H), |V(H)| - \tau(H)}$-free. By \cref{main}, $f(H)\leq\tau(H)$. 

We now establish the lower bound on $f(H)$. Observe that $G(s,N)$ is the closure of the complete $(N+1)$-ary tree of height $s$, and  $G(s,N)$ has tree-depth at most $s$. Since tree-depth is minor-monotone \citep{Sparsity}, every minor of $G(s,N)$ has tree-depth at most $s$. Thus a graph $H$ is not a minor of $G(\td(H)-1,N)$. By \cref{LowerBound}, every $(\td(H)-2)$-colouring of $G(\td(H)-1,N)$ has a colour class that induces a subgraph with maximum degree at least $N$. Thus $f(H)\geq \td(H)-1$. 
\end{proof}

The lower and upper bounds in \cref{Hfree} match in some important cases, like $H=K_t$ or $H=K_{s,t}$ (or $H=K_{s,t}^*$). The Petersen graph $P$ is an example where they do not match. \cref{Hfree} implies $f(P)\in\{5, 6\}$. On the other hand, every $P$-minor-free graph is $9$-colourable \citep{HW16} and this is best possible since $K_9$ is $P$-minor-free. So our upper bound improves the obvious bound deduced from chromatic number. Paths provide an interesting example where the bounds in \cref{Hfree} are far apart. In particular, for a path of order $2^t-1$, \cref{Hfree} gives
\begin{equation*}
t-1\leq f(P_{2^t-1})\leq 2^{t-1}-1. 
\end{equation*}

It is easy to characterise the graphs with $f(H)=1$, in which case the lower and upper bounds in \cref{Hfree} are equal. 

\begin{PROP}
\label{fH1}
$f(H)=1$ if and only if $H$ is a star plus some isolated vertices.
\end{PROP}

\begin{proof}
Say $H$ is a $k$-leaf star plus $\ell$ isolated vertices. 
Consider a graph $G$. 
If $G$ has maximum degree at most $k-1$, then $G$ is $(1,k-1)$-colourable. 
If $G$ has at most $k+\ell$ vertices, then $G$ is $(1,k+\ell-1)$-colourable. 
Otherwise, $G$ has maximum degree at least $k$ and has at least $k+1+\ell$ vertices, in which case $G$ contains $H$ as a minor. Thus every $H$-minor-free graph is $(1,k+\ell-1)$-colourable, and $f(H)=1$. Conversely, say $H$ is not a star plus some isolated vertices. Then $H$ has two disjoint edges. For each integer $d$, the $(d+1)$-leaf star has no $H$-minor and is not $(1,d)$-colourable. Thus $f(H)\geq 2$. 
\end{proof}

The upper bound in \cref{Hfree} is not tight in general. For example, if $H$ is the $k$-ary tree of height 3, then $\tau(H)=k$ but $f(H)=2$ (as proved in \cref{td3} below). These observations lead to the following conjecture:

\begin{CON}
\label{tdH}
$f(H)=\td(H)-1$ for every graph $H$, unless $H$ has distinct connected components  $H_1$ and $H_2$ with $\td(H)=\td(H_1)=\td(H_2)$, in which case $f(H)=\td(H)$.
\end{CON}

We now explain the necessity of the exception in \cref{tdH}. Suppose $H$ has connected components $H_1$ and $H_2$ with $\td(H)=\td(H_1)=\td(H_2)=s$. If $H$ is a minor of  $G(s,n)$, then only one of $H_1$ and $H_2$ can use the root vertex of $G(s,n)$, implying one of $H_1$ and $H_2$ is a minor of $G(s-1,n)$, which contradicts the tree-depth assumption. Thus, $H$ is not a minor of $G(s,n)$. By \cref{LowerBound}, the class of $H$-minor-free graphs has defective chromatic number at least $s=\td(H)$.

\cref{fH1} confirms \cref{tdH} when $f(H)=1$. We now prove the first non-trivial case.

\begin{THM}
\label{td3}
For every graph $H$ with tree-depth 3 and with at most one component of tree-depth 3, the defective chromatic number of the class of $H$-minor-free graphs equals 2. 
\end{THM}

\begin{proof}
The lower bound is proved above. For the upper bound, since at most one component of $H$ has tree-depth 3, $H$ is a subgraph of $G(2,k)$ for some integer $k\leq |V(H)|$. By \cref{kell} below with $\ell=k$, every $G(2,k)$-minor-free graph is $(2,d)$-colourable, for some increasing function $d=d(k)$. Every $H$-minor-free graph is $G(2,k)$-minor-free. Since $k\leq |V(H)|$, every $H$-minor-free graph is $(2,d)$-colourable, where $d=d(k)\leq d(|V(H)|)$.  
\end{proof}

\begin{LEM}
\label{kell}
Let $H$ be the graph obtained from $\ell$ disjoint copies of $K_{1,k}$ by adding one dominant vertex, for some $\ell\geq 2$ and $k\geq 1$ (as illustrated in \cref{H}). Then every $H$-minor-free graph $G$ is $(2,O(\ell^{10} k^3))$-colourable.
\end{LEM}

\begin{proof}
Since $H$ is $2$-degenerate, there exists $\delta< 7(\ell k+\ell+1)$ such that every $H$-minor-free graph has average degree at most $\delta$ by a result of \citet[Lemma~3.3]{ReedWood16}. Let $r:=\binom{\ell^2-1}{2}(k+1)+\ell^2+\ell$. Let $X$ be the set of vertices  $v\in V(G)$ such that $|N_G(v)\cap N_G(w)|\geq r$ for some vertex $w\in V(G)\setminus\{v\}$. Note that $w$ is also in $X$. Let $Q$ be the graph with vertex set $V(G)$ where $vw\in E(Q)$ whenever $|N_G(v)\cap N_G(w)| \geq r$. For each edge $e=vw\in E(Q)$, let $N(e):=N_G(v)\cap N_G(w)$. Thus $|N(e)|\geq r$. Let $Y:=V(G)\setminus X$.

\begin{figure}
\centering
\includegraphics{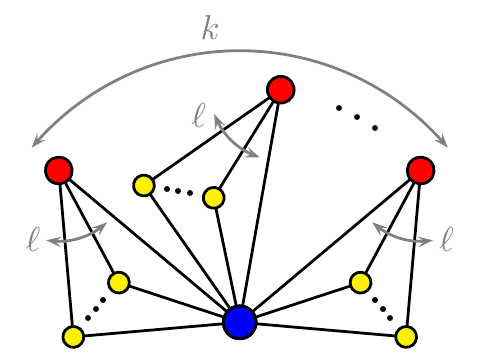}
\caption{\label{H}The graph $H$.}
\end{figure}

\begin{CLAIM}
\label{Qdegree}
$Q$ has maximum degree less than $\ell$. 	
\end{CLAIM}

\begin{clproof} 
Suppose on the contrary that some vertex $v$ in $Q$ is adjacent to distinct vertices $v_1,\dots,v_\ell$ in $Q$. For $i=1,2,\dots,\ell$, choose $k+1$ common neighbours of $v$ and $v_i$ in $G$ that have not already been chosen and are different from $v,v_1,\dots,v_\ell$. This is possible, since $v$ and $v_i$ have $r\geq (k+1) \ell + \ell-1$ common neighbours in $G$. For each $i\in[\ell]$, contract the edge between $v_i$ and one of the chosen common neighbours of $v$ and $v_i$. The chosen vertices along with $v,v_1,\dots,v_\ell$ form $H$ as a minor of $G$, which is a contradiction.
\end{clproof}

\begin{CLAIM}
\label{newclaim}
If $R$ is a set of more than $\ell(\ell-1)$ vertices in $X$, then $Q$ contains an $\ell$-edge matching, each edge of which has at least one endpoint in $R$. 
\end{CLAIM}

\begin{clproof} 
Let $Z$ be the subgraph of $Q$ induced by $R\cup N_Q(R)$. 
Label vertices in $R$ \emph{red} and vertices in $N_Q(R)\setminus R$ \emph{blue}. 
If $\Delta$ is the maximum degree of $Z$, then $\Delta\leq\ell-1$ by \cref{Qdegree}. 
The number of red vertices is $|R|> (\Delta+1)(\ell-1)$. 
Every vertex in $R$ is in $X$ and thus has a neighbour in $Q$, which is in $N_Q(R)$. 
Hence $Z$ has no red isolated vertex. 
Let $Z'$ be an edge-minimal spanning subgraph of $Z$ with no red isolated vertex. By minimality, each edge of $Z'$ has a red endpoint with degree 1. Thus each component of $Z'$ is either a blue isolated vertex, a red--blue edge, or a star with all its leaves coloured red. Since each component of $Z'$ has at most $\Delta+1$ red vertices, and there are strictly greater  than $(\Delta+1)(\ell-1)$ red vertices, $Z'$ contains at least $\ell$ non-singleton components. 
Let $v_1w_1,\dots,v_\ell w_\ell$ be a matching obtained by choosing one edge from each non-singleton component of $Z'$, where $v_1,\dots,v_\ell$ are red and thus in $R$. 
\end{clproof}

\begin{CLAIM}
\label{Xneighbours}
Every vertex is adjacent in $G$ to less than $\ell^2$ vertices in $X$. 
\end{CLAIM}

\begin{clproof}
Suppose on the contrary that $|N_G(v)\cap X|\geq\ell^2$ for some vertex $v\in V(G)$. If $v\in X$ then let $R:=(N_G(v)\cap X)\setminus N_Q(v)$, otherwise let $R:=N_G(v)\cap X$.  Then $|R|>\ell(\ell-1)$ since $|N_Q(v)| \leq \ell-1$ by \cref{Qdegree}.  By \cref{newclaim},  $Q$ contains a matching $v_1w_1,\dots,v_\ell w_\ell$, where $v_1,\dots,v_\ell$ are in $R$. By construction, each $w_i\neq v$. For $i=1,2,\dots,\ell$, choose $k+1$ common neighbours of $v_i$ and $w_i$ in $G$ that have not already been chosen and are different from $v,v_1,w_1,\dots,v_\ell,w_\ell$. This is possible, since $v_i$ and $w_i$ have $r\geq \ell(k+1)+2\ell-1$  common neighbours. For each $i\in[\ell]$, contract the edge $vv_i$ into $v$, and contract the edge between $w_i$ and one of the chosen common neighbours of $v_i$ and $w_i$. The chosen vertices along with $v,w_1,\dots,w_\ell$ form $H$ as a minor of $G$, which is a contradiction.
\end{clproof}

Let $G'$ be the graph obtained from $G$ as follows. 
For each component $C$ of $Q$, identify $V(C)$ into one vertex, and delete resulting loops and parallel edges. 
Each vertex of $G'$ corresponds to a component of $Q$. 

The final step of this proof applies \cref{main} with $s=2$ to obtain a defective 2-colouring of $G'$, from which we obtain a defective 2-colouring of $G$. To apply \cref{main} we show that $G'$ has no large $K_{2,t}$ subgraph, has  bounded $\nb$, and (thus) bounded average degree.

Consider a $K_{2,t}$ subgraph in $G'$. There are distinct components $C,D,A_1,\dots,A_t$ of $Q$,  such that for $i\in[t]$, 
some vertex in $A_i$ is adjacent in $G$ to some vertex in $C$, and
some vertex in $A_i$ is adjacent in $G$ to some vertex in $D$.
Note each $A_i$ is either a single-vertex component of $Q$ contained in $Y$ or is contained in $X$ with at least two vertices. 


\begin{CLAIM}
\label{AinX}
	$|\{i\in[t]:A_i \subseteq X\}| <\ell^2$
\end{CLAIM}

\begin{clproof}
Suppose on the contrary and without loss of generality that $A_1,\dots,A_{\ell^2}\subseteq X$. 
The component $C$ is not a single vertex, as otherwise, this vertex would have at least $\ell^2$ neighbours in $X$ contradicting  \cref{Xneighbours}. Thus $C$ is contained in $X$ and has at least two vertices.
For each $i\in[\ell]$, let $v_i$ be a vertex in $A_i$ adjacent to a vertex in $C$.
Since $v_i$ is in $A_i\subseteq X$, there is an edge $e_i=u_iv_i\in E(Q)$ and thus $u_i$ is also in $A_i$. 
Note that $u_1,\dots,u_{\ell}$ are distinct since they belong to different components $A_i$. 
Let $E(C)$ be the set of edges of $Q$ between vertices in $C$. 
Construct a bipartite graph $B$ with colour classes $$B_1:=E(C)\cup\{e^j_i:i\in[\ell],j\in[k+1]\}\text{ and }B_2:=Y,$$ 
where the vertex corresponding to each $f\in E(C)$ is adjacent to each vertex in $N(f)\setminus X$, 
and similarly  the vertex $e_i^j$ is adjacent to each vertex in $N(e_i)\setminus X$ for each $i\in [\ell]$ and $j\in [k+1]$. 
The endpoints of each edge in $Q$ have at least $r$ common neighbours in $G$,  
at most $\ell^2-1$ of which are in $X$ by \cref{Xneighbours}. 
Thus, in $B$, every vertex in $B_1$ has degree at least $r-\ell^2+1$, 
and every vertex in $B_2$ has degree at most $\binom{\ell^2-1}{2}(k+1)$ by \cref{Xneighbours}. 
Consider a subset $S\subseteq B_1$. 
The number of edges between $S$ and $N_B(S)$ is at least $(r-\ell^2+1)|S|$ and at most $\binom{\ell^2-1}{2}(k+1)|N_B(S)|$, 
implying $|N_B(S)|\geq |S|$ since $r-\ell^2+1\geq\binom{\ell^2-1}{2}(k+1)$. 
By Hall's Theorem, $B$ contains a matching with every vertex in $B_1$ matched. 
For $i\in[\ell]$ and $j\in[k+1]$, let $x_i^j$ be the vertex in $B_2$ matched with $e_i^j$.
Then $x_i^j$ is a common neighbour of $u_i$ and $v_i$ in $G$.
For each edge $f\in E(C)$, let $x_f$ be the vertex in $B_2$ matched with $f$. Then $x_f$ is a common neighbour of the endpoints of $f$ in $G$.
All these $x$-vertices are distinct and are contained in $Y$. 
Hence $V(C)\cup \{ x_f : f=pq\in E(C)\}\cup \{v_1,v_2,\ldots,v_\ell\}$ induces a connected subgraph of 
$G-(\{u_1,u_2,\dots,u_\ell\}\cup\{x_i^j:i\in[\ell],j\in[k+1]\})$; contract this connected subgraph into a vertex $z$. 
Now $z$ is adjacent to $x_i^j$ for each $i\in[\ell]$ and $j\in[k+1]$. 
Finally, contract the edge $u_ix_i^{k+1}$ into $u_i$, for each $i\in[\ell]$. 
Now $z$ is adjacent to $u_1,\dots,u_\ell$, and 
$x_i^1,\dots,x_i^k$ are common neighbours of $z$ and $u_i$. 
Hence $H$ is a minor of $G$, which is a contradiction.
\end{clproof}

\begin{CLAIM}
\label{AinY}
	$|\{i\in[t]:A_i \subseteq Y\}| \le \ell^2 (\ell-1)^2 (r-1)$.
\end{CLAIM}

\begin{clproof}
Define $Z:= \bigcup\{A_i:i\in[t], A_i\subseteq Y\}$. 
Note that $|Z|=|\{i\in[t]:A_i \subseteq Y\}|$ because $\lvert A_i\rvert =1$ if $A_i\subseteq Y$.
Let $C'$ be the set of vertices in $C$ with some neighbour in $Z$. 
Let $D'$ be the set of vertices in $D$ with some neighbour in $Z$. 
For $v\in V(C')$ and $w\in V(D')$, less than $r$ vertices are common neighbours of $v$ and $w$ (since $vw\not\in E(Q)$). 
Thus $|Z| \le |C'|\,|D'|\,(r-1)$, and we are done if $|C'|\le \ell(\ell-1)$ and $|D'|\le\ell(\ell-1)$. 

Suppose for the sake of contradiction and without loss of generality that $|C'|> \ell(\ell-1)$. 
Then $V(C)\subseteq X$ since $\ell(\ell-1)\geq 2$.
By \cref{newclaim} with $R=C'$, there is a matching $e_1,\dots,e_\ell$ in $Q$, where $e_i=v_iu_i$ and $v_i\in C'$ for each $i\in[\ell]$. 
For $i\in[\ell]$, let $a_i$ be a (not necessarily distinct) neighbour of $v_i$ in $Z$. 
Let $E(D)$ be the set of edges in $Q$ between vertices in $D$. 
Construct a bipartite graph $B$ with colour classes 
\begin{align*}
B_1 &:=E(D)\cup \{e^j_i:i\in[\ell],j\in[k+1]\} \text{ and }\\
B_2 & :=Y\setminus (\{a_1,\dots, a_\ell\}\cup V(D)),  
\end{align*}
where the vertex in $B_1$ corresponding to each edge $f\in E(D)$ is adjacent to each vertex in $N(f)\setminus(X\cup V(D)\cup \{a_1,\dots, a_\ell\})$, and similarly the vertex in $B_1$ corresponding to each edge $e_i^j$ is adjacent to each vertex in $N(e_i)\setminus(X\cup V(D)\cup \{a_1,\dots, a_\ell\})$.
Note that $|Y\cap V(D)|\le 1$ and if $E(D)\neq\emptyset$, then $V(D)\subseteq X$.
The endpoints of each edge in $Q$ have at least $r$ common neighbours in $G$,  
at most $(\ell^2-1)+\ell+1$  of which are in $X\cup V(D)\cup \{a_1,\dots,a_\ell\}$ by \cref{Xneighbours}. 
Thus, in $B$, every vertex in $B_1$ has degree at least $r-\ell^2-\ell$, 
and every vertex in $B_2$ has degree at most $\binom{\ell^2-1}{2}(k+1)$ by \cref{Xneighbours}.
Consider a subset $S\subseteq B_1$. 
The number of edges between $S$ and $N_B(S)$ is at least $(r-\ell^2-\ell)|S|$ and at most $\binom{\ell^2-1}{2}(k+1)|N_B(S)|$, 
implying $|N_B(S)|\geq |S|$ since $r-\ell^2-\ell\geq\binom{\ell^2-1}{2}(k+1)$. 
By Hall's Theorem, $B$ contains a matching with every vertex in $B_1$ matched. 

For $f\in E(D)$, let $x_f$ be the vertex in $B_2$ matched with $f$
and for each $i\in [\ell]$ and $j\in[k+1]$, let $x_i^j$ be the vertex in $B_2$ matched with $e_i^j$.
Then $x_f$ is a common neighbour (in $G$) of the endpoints of $f$ and $x_i^j$ is a common neighbour of $u_i$ and $v_i$ in $G$.
Note that all these $x$-vertices are distinct and are in $V(G)\setminus (X\cup V(D)\cup \{a_1,\dots,a_\ell\})$. 
Each vertex $a_i$ has a neighbour in $D$. 
Hence $\{a_1,\dots,a_\ell,v_1,\ldots,v_\ell\}\cup V(D)\cup \{ x_f : f\in E(D)\}$ induces a connected subgraph of 
$G-(\{u_1,u_2,\dots,u_\ell\}\cup\{x_i^j:i\in[\ell],j\in[k+1]\})$. 
Contract this connected subgraph into a vertex $z$. 
Now $z$ is adjacent to $x_i^j$ for each $i\in[\ell]$ and $j\in[k+1]$. 
Finally, contract the edge $u_ix_i^{k+1}$ into $u_i$, for each $i\in[\ell]$. 
Now $z$ is adjacent to $u_1,\dots,u_\ell$, and 
$x_i^1,\dots,x_i^k$ are common neighbours of $z$ and $u_i$ for each $i\in [\ell]$. 
Hence $H$ is a minor of $G$, which is a contradiction.
\end{clproof}

\cref{AinX,AinY} show that $t<\ell^2+\ell^2(\ell-1)^2(r-1)\le \ell^2(\ell-1)^2r$. That is, $G'$ has no $K_{2,\ell^2(\ell-1)^2r}$ subgraph.  

\begin{CLAIM}
\label{nbGG}
	$\nb(G') \leq \delta/2 + \ell-1$.
\end{CLAIM}

\begin{clproof}
Suppose that a $(\leq 1)$-subdivision of some graph $G''$ is a subgraph of $G'$. 
Let $X''$ be the set of vertices of $G''$ that arise from components of $Q$ contained in $X$. 

Assume for contradiction that some vertex in $G''$ has at least $\ell$ neighbours in $X''$. That is, there are distinct components $C,C_1,\dots,C_{\ell}$ of $Q$, such that for each $i\in[\ell]$, $C_i$ is a non-singleton component of $Q$ contained in $X$, and there exists an edge joining a vertex of $C$ to a vertex of $C_i$ in $G$, or  a component $D_i$ of $Q$ having a neighbour of $C$ and a neighbour of $C_i$ in $G$.
If $C$ and $C_i$ are joined by an edge in $G$, then let $D_i:=C$ (for convenience). 
Note that $C$ might be a singleton component of $Q$ contained in $Y$ or a non-singleton component contained in $X$, 
and similarly for $D_1,\dots,D_\ell$, 
but $C_1,\dots,C_\ell$ are non-singleton components of $Q$. 
Let $Y'=Y\cap (V(C) \cup V(D_1)\cup \dots\cup V(D_\ell))$. Then $|Y'|\leq \ell+1$. 
For each $i\in [\ell]$, let $v_i$ be a vertex in $C_i$ adjacent to some vertex in $V(C)\cup V(D_i)$. 
Since $v_i$ is in $X$, there is an edge $e_i=u_iv_i\in E(Q)$ and thus $u_i$ is also in $C_i$. 
Construct a bipartite graph $B$ with colour classes 
\begin{align*}
B_1 & :=E(C)\cup E(D_1)\cup\dots\cup E(D_\ell)\cup\{e^j_i:i\in[\ell],j\in[k+1]\}\text{ and }\\
B_2 & :=V(G)\setminus (X\cup Y'),
\end{align*}
where the vertex corresponding to each edge $f\in E(C)\cup E(D_1)\cup\dots\cup  E(D_\ell)$ is adjacent to each vertex in $N(f)\setminus (X\cup Y')$, 
and similarly the vertex $e_i^j$ is adjacent to each vertex in $N(e_i)\setminus (X\cup Y')$.
The endpoints of each edge in $Q$ have at least $r$ common neighbours in $G$,  
at most $\ell^2-1+|Y'|$ of which are in $X\cup Y'$ by \cref{Xneighbours}.
Thus, in $B$, every vertex in $B_1$ has degree at least $r-(\ell^2-1+|Y'|)\geq r-\ell^2-\ell$, 
and every vertex in $B_2$ has degree at most $\binom{\ell^2-1}{2}(k+1)$ by \cref{Xneighbours}. 
Consider a subset $S\subseteq B_1$. 
The number of edges between $S$ and $N_B(S)$ is at least $(r-\ell^2-\ell)|S|$ and at most $\binom{\ell^2-1}{2}(k+1)|N_B(S)|$, 
implying $|N_B(S)|\geq |S|$ since $r-\ell^2-\ell\geq\binom{\ell^2-1}{2}(k+1)$. 

By Hall's Theorem, $B$ contains a matching with every vertex in $B_1$ matched. 
For $i\in[\ell]$ and $j\in[k+1]$, let $x_i^j$ be the vertex in $B_2$ matched with $e_i^j$. Then $x_i^j$ is a common neighbour of $u_i$ and $v_i$ in $G$.
For each edge $f\in E(C)\cup E(D_1)\cup\dots\cup E(D_\ell)$, let $x_f$ be the vertex in $B_2$ matched with $f$.
Then $x_f$ is a common neighbour (in $G$) of the endpoints of $f$.
All these $x$-vertices are distinct and are contained in $V(G)\setminus (X\cup Y')$. 
Hence 
\begin{multline*}
V(C)\cup V(D_1)\dots\cup V(D_\ell)\cup \{v_1,v_2,\ldots,v_\ell\}\\
\cup \{ x_f : f\in E(C)\cup E(D_1)\cup\dots\cup E(D_\ell)\}
\end{multline*}
induces a connected subgraph of 
$G-(\{u_1,u_2,\dots,u_\ell\}\cup\{x_i^j:i\in[\ell],j\in[k+1]\})$; contract this connected subgraph into a vertex $z$. 
Now $z$ is adjacent to $x_i^j$ for each $i\in[\ell]$ and $j\in[k+1]$. 
Finally, contract the edge $u_ix_i^{k+1}$ into $u_i$ for each $i\in[\ell]$. 
Now $z$ is adjacent to $u_1,\dots,u_\ell$, and 
$x_i^1,\dots,x_i^k$ are common neighbours of $z$ and $u_i$ for each $i\in [\ell]$.
Hence $H$ is a minor of $G$. 
This contradiction proves that each vertex in $G''$ has less than $\ell$ neighbours in $X''$. 

Thus $G''$ contains at most $(\ell-1)|V(G'')|$ edges with at least one endpoint in $X''$. 
Since $G''-X''$ is a subgraph of $G$, the average degree of $G''-X''$ is at most $\delta$, and $|E(G''-X'')|\le \delta|Y'|/2$. In total, $|E(G'')| \leq  \delta|Y'|/2 +(\ell-1)|V(G'')|  \leq (\delta/2+\ell-1)|V(G'')| $, and $\nb(G')\leq \delta/2 + \ell-1$.
\end{clproof}

Note that \cref{nbGG} implies $\mad(G)=2\nb[0](G)\leq 2\nb(G)\leq \delta + 2\ell-2$. By \cref{main}, $G'$ is $(2,d')$-colourable, where 
\[
d'=\lfloor N_1(2,\ell^2(\ell-1)^2r,\delta+2\ell-2, \delta + 2\ell-2)\rfloor -1.
\] 
Colour each vertex $v$ of $G$ by the colour assigned to the vertex of $G'$ corresponding to the component of $Q$ containing $v$.
Then, for each vertex $v$ of $C$, the number of neighbours of $v$ having the same colour as $v$ is at most $d'+\ell^2-1$, because $v$ has at most $\ell^2-1$ neighbours in $X$ by \cref{Xneighbours} and at most $d'$ neighbours of the same colour in $Y$.
Therefore $G$ is $(2,d'+\ell^2-1)$-colourable. We now estimate the degree bound. We have
$r\leq O(\ell^4k)$ and $\delta+2\ell-2\leq O(\ell k)$.
Thus
$d' \leq O((\delta+2\ell)^2 \ell^2(\ell-1)^2r ) \leq O(\ell^{10}k^3)$.
Therefore $G$ is $(2, O(\ell^{10}k^3))$-colourable. 
\end{proof}

Our final result provides further evidence for \cref{tdH}. It concerns graphs that exclude a fixed tree as a subgraph.

\begin{PROP}
\label{ExcludeTreeSubgraph}
Let $T$ be a tree with $n\geq 2$ vertices and radius $r\geq 1$. 
Then every graph containing no $T$ subgraph is $(r,n-2)$-colourable.
\end{PROP}

\begin{proof}
For $i=1,2,\dots,r-1$, let $V_i$ be the set of vertices $v\in V(G)\setminus(V_1\cup\dots\cup V_{i-1})$ that have at most $n-2$ neighbours in $V(G)\setminus(V_1\cup\dots\cup V_{i-1})$. Let $V_r:=V(G)\setminus(V_1\cup\dots\cup V_{r-1})$. Then $V_1\cup\dots\cup V_r$ is a partition of $V(G)$. For $i\in[1,r-1]$, by construction, $G[V_i]$ has maximum degree at most $n-2$, as desired. Suppose that $G[V_r]$ has maximum degree at least $n-1$. We now show that $T$ is a subgraph of $G$, where each vertex $v$ of $T$ is mapped to a vertex $v'$ of $G$. Let $x$ be the centre of $T$. Map the vertices of $T$ to vertices in $G$ in order of their distance from $x$ in $T$, where  $x$ is mapped to a vertex $x'$ with degree at least $n-1$ in $G[V_r]$.  The key invariant is that each vertex $v$ at distance $i$ from $x$ in $T$ is mapped to a vertex $v'$ in $V_{r-i+1}\cup\dots\cup V_r$. If $i=0$ then $v=x$ and by assumption, $v'$ has at least $n-1$ neighbours in $V_r$. If $i\in[1,r-1]$ then by construction, $v'$ has at least $n-1$ neighbours in $V_{r-i}\cup\dots\cup V_r$ (otherwise $v'$ would be in $V_{r-i}$). Thus there are always unmapped vertices in $V_{r-i}\cup\dots\cup V_r$ to choose as the children of $v$. Hence $T$ is a subgraph. This contradiction shows that 
$G[V_r]$ has maximum degree at least $n-2$, and $G$ is $(r,n-2)$-colourable. 
\end{proof}


Note that \cref{ExcludeTreeSubgraph} is best possible for the complete binary tree $T$ of radius $r$, which has tree-depth $r+1$ (see \citep[Exercise~6.1]{Sparsity}). Thus $G(r,N)$ contains no $T$ subgraph, and \cref{LowerBound,ExcludeTreeSubgraph} imply that the defective chromatic number of the class of graphs containing no $T$ subgraph equals $r$.

Note that the behaviour shown in \cref{ExcludeTreeSubgraph} is qualitatively different from the chromatic number of graphs excluding a given tree as a subgraph. Say $T$ is a tree with $n$ vertices. A well known greedy embedding procedure shows that every graph with minimum degree at least $n-1$ contains $T$ as a subgraph. That is, every graph containing no $T$ subgraph is $(n-2)$-degenerate, and is thus $(n-1)$-colourable. This bound is tight since $K_{n-1}$ contains no $T$ subgraph and is $(n-1)$-chromatic. In short, for the class of graphs containing no $T$ subgraph, the chromatic number equals $n-1$, whereas \cref{ExcludeTreeSubgraph} says that the defective chromatic number is at most the radius of $T$. 

\cref{tdH} suggests similar behaviour for $H$-minor-free graphs. Say $H$ has $n$ vertices. Hadwiger's Conjecture says that the maximum chromatic number of the class of $H$-minor-free graphs equals $n-1$. It is at least $n-1$ since $K_{n-1}$ is $H$-minor-free, and at most $O(n\sqrt{\log n})$ in general.   \cref{tdH} says that if $H$ is connected, then the defective chromatic number of the class of $H$-minor-free graphs equals the tree-depth of $H$ minus 1.


\def\soft#1{\leavevmode\setbox0=\hbox{h}\dimen7=\ht0\advance \dimen7
  by-1ex\relax\if t#1\relax\rlap{\raise.6\dimen7
  \hbox{\kern.3ex\char'47}}#1\relax\else\if T#1\relax
  \rlap{\raise.5\dimen7\hbox{\kern1.3ex\char'47}}#1\relax \else\if
  d#1\relax\rlap{\raise.5\dimen7\hbox{\kern.9ex \char'47}}#1\relax\else\if
  D#1\relax\rlap{\raise.5\dimen7 \hbox{\kern1.4ex\char'47}}#1\relax\else\if
  l#1\relax \rlap{\raise.5\dimen7\hbox{\kern.4ex\char'47}}#1\relax \else\if
  L#1\relax\rlap{\raise.5\dimen7\hbox{\kern.7ex
  \char'47}}#1\relax\else\message{accent \string\soft \space #1 not
  defined!}#1\relax\fi\fi\fi\fi\fi\fi}

\parindent 0mm
\end{document}